\documentclass[10pt, twoside, reqno]{amsart}
\usepackage{xcolor}
\usepackage{amstext}
\usepackage{amsmath}
\usepackage{amsfonts}
\usepackage{amssymb}
\usepackage{amsbsy}
\usepackage{latexsym}
\usepackage[T1]{fontenc}
\usepackage{xy}
\usepackage{hhline}
\xyoption{all}
\newcounter{num}[section] %

\newenvironment{theo}
{\refstepcounter{num}%
\bigskip\noindent{\bf Theorem~\arabic{section}.\arabic{num}. }\it}
{\smallskip}

\newenvironment{cor}
{\refstepcounter{num}%
\bigskip\noindent{\bf Corollary~\arabic{section}.\arabic{num}. }\it}

\newenvironment{lemma}
{\refstepcounter{num}%
\bigskip\noindent{\bf Lemma~\arabic{section}.\arabic{num}. }\it}

\newenvironment{conj}
{\refstepcounter{num}%
\bigskip\noindent{\bf Conjecture~\arabic{section}.\arabic{num}. }\it}

\newenvironment{eq}{\begin{equation}}{\end{equation}}

\renewcommand{\Ref}[1]{(\ref{#1})}

\newcommand{\si}{\sigma}
\newcommand{\al}{\alpha}
\newcommand{\be}{\beta}
\newcommand{\ga}{\gamma}
\newcommand{\la}{\lambda}
\newcommand{\de}{\delta}

\newcommand{\ov}[1]{\overline{#1}}
\newcommand{\un}[1]{{\underline{#1}} }

\newcommand{\tr}{\mathop{\rm tr}}

\newcommand{\mdeg}{\mathop{\rm mdeg}}

\newcommand{\Char}{\mathop{\rm char}}

\newcommand{\Ker}{{\mathop{\rm{Ker }}}}

\newcommand{\Sym}{{\mathcal S}}
\newcommand{\wz}{\mathop{\rm wz}}
\newcommand{\upto}{,\ldots ,}

\newcommand{\GL}{{\rm GL}}         
\newcommand{\SL}{{\rm SL}}         

\newcommand{\N}{\mathcal{N}} 

\newcommand{\FF}{{\mathbb{F}}}   
\newcommand{\KK}{{\mathbb{K}}}   
\newcommand{\CC}{{\mathbb{C}}}   
\newcommand{\NN}{{\mathbb{N}}}

\renewcommand{\k}{\kappa}

\newcommand{\matr}[4]{\left(
\begin{array}{cc}
#1 & #2 \\ 
#3 & #4 \\ 
\end{array}
\right)}

\newcommand{\besep}{\be_{\rm sep}}
\newcommand{\sisep}{\si_{\rm sep}}

\newcommand{\mylabel}[1]{}

\begin{document}
\title[On $m$-tuples of nilpotent $2\times2$ matrices over an arbitrary field]{On $m$-tuples of nilpotent $2\times2$ matrices over an arbitrary field}

\thanks{The work was supported by FAPESP 2018/23690-6}

\author{Artem Lopatin}
\address{Artem Lopatin\\ 
Universidade Estadual de Campinas (UNICAMP), 651 Sergio Buarque de Holanda, 13083-859 Campinas, SP, Brazil}
\email{dr.artem.lopatin@gmail.com (Artem Lopatin)}


\begin{abstract} The algebra of $\GL_n$-invariants of $m$-tuples of $n\times n$ matrices with respect to the action by simultaneous conjugation is a classical topic in case of infinite base field. On the other hand, in case of a finite field generators of polynomial invariants even  in case of a pair of $2\times 2$ matrices are not known.   Working over an arbitrary field we classified all $\GL_2$-orbits on $m$-tuples of $2\times 2$ nilpotent matrices for all $m>0$. As a consequence, we obtained a minimal separating set for the algebra of $\GL_2$-invariant polynomial functions of $m$-tuples of $2\times 2$ nilpotent matrices. We also described the least possible number of elements of a separating set for an algebra of invariant polynomial functions over a finite field.  
 

\end{abstract}

\maketitle

\section{Introduction}\label{section_intro}

\subsection{Algebras of invariants}\label{section_inv} All vector spaces, algebras, and modules are over an arbitrary (possibly finite) field $\FF$ of arbitrary characteristic $p\geq0$ unless otherwise stated.  By an algebra we always mean an associative algebra with unity. 

Consider an $n$-dimensional vector space $V$ over the field $\FF$ with a fixed basis $v_1,\ldots,v_n$.  The coordinate ring $\FF[V]  = \FF[ x_1 \upto x_n]$ of $V$ is isomorphic to the symmetric algebra $S(V^{\ast})$ over the dual space $V^{\ast}$, where $x_1 \upto x_n$ is the basis for $V^\ast$.  Let $G$ be a subgroup of $\GL(V) \cong \GL_n(\FF)$. The algebra $\FF[V]$ becomes a $G$-module with
\begin{eq}\label{eq_action}
(g \cdot f)(v) = f(g^{-1}\cdot v) 
\end{eq}%
for all $f \in V^\ast$ and $v \in V$. The algebra of {\it polynomial invariants} is defined as follows:
$$ \FF[V]^G = \{ f \in \FF[V] \mid g\cdot f=f  \, \text{ for all  } g \in G \}.  
$$

The algebra of polynomial functions $\mathcal{O}(V) = \FF[\ov{x}_1,\ldots,\ov{x}_n]$, where $\ov{x}_i:v \mapsto v_i$ for $v=(v_1,\ldots,v_n)\in V$ with respect to the fixed basis of $V$.  The algebra ${\mathcal{O}}(V)$ becomes a $G$-module by formula~\Ref{eq_action}
for all $f \in {\mathcal{O}}(V)$ and $v \in V$. The algebra of {\it invariant polynomial functions} is defined as follows:
\begin{eqnarray*} {\mathcal{O}}(V)^G &=& \{ f \in {\mathcal{O}}(V) \mid g\cdot f=f  \, \text{ for all  } g \in G \} \\ & =& \{ f \in {\mathcal{O}}(V) \mid f(g\cdot v)=f(v) \, \text{ for all } g\in G, v \in V \}. \end{eqnarray*}
Obviously, $\mathcal{O}(V)\simeq \FF[V] / \Ker(\Psi)$ for the homomorphism of algebras $\Psi: \FF[V]\to \mathcal{O}(V)$ defined by $x_i\to \ov{x}_i$. In case $\FF$ is infinite we have that $\Ker(\Psi)=0$ and $\FF[V]^G = \mathcal{O}(V)^G$. If $\FF=\FF_q$ is a finite field, then the ideal $\Ker(\Psi)$ is generated by $x_i^q - x_i$ for all $1 \leq i \leq n$. 
Using $\Psi$ we can consider any element $f\in\FF[V]$ as the map $V\to \FF$. Namely, we write $f(v)$ for $\Psi(f)(v)$, where $v\in V$.  For any infinite extension  $\FF\subset \KK$  we denote ${V}_{\KK} = V \otimes_{\FF} \KK$. Therefore, 
\begin{eqnarray*} \FF[V]^G &= & \{ f \in \FF[V] \mid f(g\cdot v)=f(v) \, \text{ for all } g\in G, v \in {V}_{\KK} \}  \\
& \subset & \{ f \in \FF[V] \mid f(g\cdot v)=f(v) \, \text{ for all } g\in G, v \in V \} \end{eqnarray*}

Assume that $W\subset V$ is a $G$-invariant subset of $V$. The algebra of polynomial functions $\mathcal{O}(W)$ on $W$ is generated by functions $y_1,\ldots,y_n:W\to \FF$, where $y_i$ is the restriction of $\ov{x}_i:V\to \FF$ to $W$. The algebra ${\mathcal{O}}(W)$ becomes a $G$-module by formula~\Ref{eq_action}
for all $f \in \mathcal{O}(W)$ and $v \in W$.  Denote by $I(W)$ the ideal of all $f\in  \mathcal{O}(W)$ that are zeros over $W$. Then we have $\mathcal{O}(W)\simeq \mathcal{O}(V) / I(W)$ and $y_i=\ov{x}_i + I(W)$. The algebra ${\mathcal{O}}(W)^G$ of {\it invariant polynomial functions on $W$}  is defined in the same way as  ${\mathcal{O}}(V)^G$. 
The {\it degree} of $f\in \mathcal{O}(W)$ is defined as the minimal degree of a polynomial $h\in\FF[V]$  with $f=\Psi(h)+I(W)$.

\subsection{Separating invariants}\label{section_separ}
In 2002 Derksen and Kemper~\cite{DerksenKemper_book} (see~\cite{DerksenKemper_bookII} for the second edition) introduced the notion of separating invariants as a weaker concept than generating invariants.   Given a subset $S$ of $\mathcal{O}(W)^{G}$, we say that elements $u,v$ of $W$ {\it are separated by $S$} if  exists an invariant $f\in S$ with $f(u)\neq f(v)$. If  $u,v\in W$ are separated by $\mathcal{O}(W)^{G}$, then we simply say that they {\it are separated}. A subset $S\subset \mathcal{O}(W)^{G}$ of the invariant ring is called {\it separating} if for any $v, w$ from $W$ that are separated we have that they are separated by $S$. A separating subset for $\FF[V]^{G}$ is defined in the same manner. We say that a separating set is minimal if it is minimal w.r.t.~inclusion. Obviously, any generating set is also separating. Denote by $\besep(\mathcal{O}(W)^G)$ the minimal integer $\besep$ such that the set of all invariant polynomial functions of degree  less or equal to $\besep$ is separating for $\mathcal{O}(W)^G$. Minimal separating sets for different actions were constructed in~\cite{Cavalcante_Lopatin_1, Domokos2020, Domokos2020Add, Lopatin_separating2x2, Kemper_Lopatin_Reimers_2, Lopatin_Reimers_1, Reimers20}.

Separating invariants for $\FF[V]^G$ in case of a finite field $\FF = \FF_q$ with $q$ elements were studied by Kemper, Lopatin, Reimers in~\cite{Kemper_Lopatin_Reimers_2}. Namely, it was shown that the minimal number of separating invariants for $\FF[V]^G$ is $\ga(q,\k) =\lceil\log_q(\k)\rceil$, where $\k$ stands for the number of $G$-orbits in $V$, and an explicit construction of a minimal separating set of $\ga(q,\k)$ invariants of degree at most $|G|n(q-1)$ was given. Moreover,  a minimal separating set of $\ga(q,\k)$ invariants of degree $\leq n(q-1)$ in the non-modular case was constructed  as well as in the case when $G$ consists entirely of monomial matrices.

\subsection{Matrix invariants}\label{section_matrix}
For $n>1$ and $m\geq 1$, the direct sum $M_n^m$ of $m$ copies of the space of $n\times n$ matrices over $\FF$ is a $\GL_n$-module with respect to the diagonal action by conjugation: $g\cdot\un{A} = (g A_1 g^{-1},\ldots,g A_d g^{-1})$ for $g\in \GL_n$ and $\un{A}=(A_1,\ldots,A_d)$ from $M_n^m$.  Given a matrix $A$, denote by $A_{ij}$ the $(i,j)^{\rm th}$ entry of $A$. Any element of the coordinate ring 
$$
\FF[M_n^m]^{\GL_n}= \FF[x_{ij}(k) \;| \;1\leq i,j\leq n, 1\leq k\leq m]
$$
of $M_n^m$ can be considered as the polynomial function $x_{ij}(k): M_n^m\to \FF$, which sends $(A_1,\ldots,A_m)$ to  $(A_k)_{ij}$.  The algebra of invariant polynomial functions $\mathcal{O}(M_n^m)^{\GL_n}$ is defined as  in Section~\ref{section_inv}.

Sibirskii~\cite{Sibirskii_1968}, Procesi~\cite{Procesi76} and Donkin~\cite{Donkin92a} established that the algebra of invariant polynomial functions $\mathcal{O}(M_n^m)^{\GL_n}$ is generated by $\si_t(X_{k_1}\cdots X_{k_r})$ in case $\FF$ is infinite, where $X_k$ stands for the $n\times n$ {\it generic} matrix $(x_{ij}(k))_{1\leq i,j\leq n}$ and $\si_t$ stands for the $t^{\rm th}$ coefficient of the characteristic polynomial, i.e., $\det(\la E - A)=\sum_{t=0}^n (-1)^{t}\la^{n-t} \si_t(A)$ for any $n\times n$ matrix $A$ over a commutative ring. In particular, $\si_0(A)=1$, $\si_1(A)=\tr(A)$ and $\si_n(A)=\det(A)$. 

For a monomial $c\in \FF[M_n^m]$ denote by $\deg{c}\in \NN$ its {\it degree} and by $\mdeg{c}\in \NN^m$ its {\it multidegree}, where $\NN$ stands for the set of non-negative integers. Namely, $\mdeg{c}=(t_1,\ldots,t_d)$, where $t_k$ is the total degree of the monomial $c$ in $x_{ij}(k)$, $1\leq i,j\leq n$, and $\deg{c}=t_1+\cdots+t_m$.   

%

 
In case of an infinite field of arbitrary characteristic the following minimal separating set for $\mathcal{O}(M_2^m)^{\GL_2}$ was given by Kaygorodov, Lopatin, Popov~\cite{Lopatin_separating2x2}:
\begin{eq}\label{eq0}
\tr(X_k^2),\, 1\leq k\leq m;\; 
\tr(X_{k_1}\cdots X_{k_r}),\, r\in\{1,2,3\},\, 1\leq k_1<\cdots<k_r\leq m.
\end{eq}%
Note that set~\Ref{eq0} generates the algebra $\mathcal{O}(M_2^m)^{\GL_2}$ if and only if the characteristic of $\FF$ is different from two or $m\leq 3$ (see~\cite{Procesi_1984,DKZ_2002}). A minimal generating set for $\mathcal{O}(M_3^m)^{\GL_3}$ was given by Lopatin in~\cite{Lopatin_Sib, Lopatin_Comm1,Lopatin_Comm2} in the case of an arbitrary infinite field. Over a field of characteristic zero a minimal generating set for $\mathcal{O}(M_n^2)^{\GL_n}$ was established by Drensky and Sadikova~\cite{Drensky_Sadikova_4x4} in case $n=4$ (see also~\cite{Teranishi_1986}) and by \DJ{}okovi\'c~\cite{Djokovic2007} in case $n=5$ (see also~\cite{Djokovic2009}). Some upper bounds on degrees of generating and separating invarints for $\mathcal{O}(M_n^m)^{\GL_n}$ were given by  Derksen and Makam in~\cite{DerksenMakam4, DerksenMakam5}. A minimal separating set for the algebra of matrix semi-invariants $\mathcal{O}(M_2^m)^{\SL_2\times \SL_2}$ was explicitly described by Domokos~\cite{Domokos2020} over an arbitrary algebraically closed field. Note that a minimal generating set for $\mathcal{O}(M_2^m)^{\SL_2\times \SL_2}$ was given by Lopatin~\cite{Lopatin_semi2222} over an arbitrary infinite field (see also~\cite{Domokos2020Add}). Elmer~\cite{Elmer_2023, Elmer_semi_2023} obtained low bounds on the least possible number of elements for separating sets for $\mathcal{O}(M_2^m)^{\GL_2}$ and $\mathcal{O}(M_2^m)^{\SL_2\times \SL_2}$ in case $\FF=\CC$.  More results on separating invariants of matrices can be found in~\cite{Lopatin_Ferreira}.

\subsection{Nilpotent matrices}\label{section_nilp}

Denote by $\N_n$ the set of all $n\times n$ {\it nilpotent} matrices, i.e., $A\in M_n$ belongs to $\N_n$ if and only if $\si_1(A)=\si_2(A)=\cdots=\si_n(A)=0$, or equivalently, $A^n=0$. 

Over a finite field, the number of $\GL_n$-orbits on $\N_n^m$ was studied by Hua~\cite{Hua_2021}. The variety of pairs of commuting nilpotent matrices has extensively been studied over the past forty years (see~\cite{Baranovsky_2001, Basili_2003, Basili_Iarrobino_2008, Bondarenko_Futorny_Petravchuk_Sergeichuk_2021, Hua_2023, Hua_4x4_2023} for more details).  

The algebra of polynomial functions $\mathcal{O}(\N_n^m)$ of $\N_n^m\subset M_n^m$  is generated by $y_{ij}(k)$ for $1\leq i,j\leq n$ and $1\leq k\leq m$, where  $y_{ij}(k):\N_n^m\to \FF$ sends $(A_1,\ldots,A_m)$ to $(A_k)_{ij}$. As in Section~\ref{section_inv} we have  $\mathcal{O}(\N_n^m)=\mathcal{O}(M_n^m)/ I(\N_n^m)$. The algebra $\mathcal{O}(\N_n^m)^{\GL_n}$ of {\it invariant polynomial functions on nilpotent matrices} is defined in the same way as $\mathcal{O}(M_n^m)^{\GL_n}$.
Note that $\si_t(Y_{k_1}\cdots Y_{k_r})$ lies in  $\mathcal{O}(\N_n^m)^{\GL_n}$, where $Y_k=(y_{ij}(k))_{1\leq i,j\leq n}$ stands for the {\it generic nilpotent} $n\times n$ matrix.   Obviously, $\si_t(Y_k^s)=0$ for all $1\leq t\leq n$, $s>0$ and $Y_k^n=0$. It is easy to see that in case of algebraically closed field of zero characteristic the algebra $\mathcal{O}(\N_n^m)^{\GL_n}$ is generated by $\tr(Y_{i_1}\cdots Y_{i_r})$ (for example, see~\cite{Cavalcante_Lopatin_1}), but in general case generators for $\mathcal{O}(\N_n^m)^{\GL_n}$ are not known for $n,m\geq2$. Moreover, over a finite field $\FF$ generators for algebras $\mathcal{O}(M_n^m)^{\GL_n}$ and $\FF[M_n^m]^{\GL_n}$ are not known as well. 

Working over an algebraically closed field of zero characteristic Cavalcante and Lopatin in~\cite{Cavalcante_Lopatin_1} showed that the set
$$S_{2,m}=\{\tr(Y_i Y_j), \; 1\leq i<j\leq m;\;\; \tr(Y_i Y_j Y_k), \; 1\leq i<j<k\leq m\}$$
is a minimal generating set and a minimal separating set for the algebra $\mathcal{O}(\N_2^m)^{\GL_2}$ for $m>0$. 
Minimal generating sets and  minimal separating sets for the algebras $\mathcal{O}(\N_3^2)^{\GL_3}$ and  $\mathcal{O}(\N_3^3)^{\GL_3}$ were also obtained  in~\cite{Cavalcante_Lopatin_1} in case $p=0$.


\subsection{Results}\label{section_results}

For $\FF=\FF_q$ in Section~\ref{section_minimal} we extend the results from~\cite{Kemper_Lopatin_Reimers_2} given in Section~\ref{section_separ} to the algebra $\mathcal{O}(W)^G$ of $G$-invariant polynomial functions on a $G$-invariant subset $W\subset V$ (see Theorem~\ref{theo1}).  Note that a separating set for $\mathcal{O}(W)^{G}$ separates all $G$-orbits on $W$ in case $\FF$ is finite (for example, it follows from the proof of Theorem~\ref{theo1}).

In Section~\ref{section_main1} working over an arbitrary field, we explicitly describe a minimal set of representatives of all $\GL_2$-orbits on $\N_2^m$ in Theorem~\ref{theo_orb}. As a consequence, in case of $\FF=\FF_q$ we explicitly calculate the number of orbits in Corollary~\ref{cor_cor} as well as the least possible number of elements for a separating set for  $\mathcal{O}(\N_2^2)^{\GL_2}$. We formulate Conjecture~\ref{conj} about the number of $\GL_n$-orbits on $\N_n^m$. 

To formulate results about separating sets, introduce some notations. For an arbitrary field $\FF$ denote
$$S_{2,m}^{(2)}=\{\tr(Y_i Y_j), \; 1\leq i<j\leq m\}.$$

Now assume that $\FF$ is finite. For all $1\leq i,j\leq n$ and $\al\in\FF^{\times}$ define $\zeta(Y_i)$ and $\eta_{\al}(Y_i,Y_j)$ from $\mathcal{O}(\N_2^m)^{\GL_2}$ as follows:
\begin{eq}\label{eq_zeta}
\zeta(A)=\left\{
\begin{array}{rl}
0, & \text{if } A\neq0 \\
1, & \text{if } A=0 \\
\end{array}
\right.,
\end{eq}
\begin{eq}\label{eq_eta}
\eta_{\al}(A,B)=\left\{
\begin{array}{rl}
0, & \text{if } \al A\neq B \\
1, & \text{if } \al A=B \\
\end{array}
\right.,
\end{eq}%
where $A,B\in\N_2$.   The presentation of $\zeta(Y_i)$ and $\eta_{\al}(Y_i,Y_j)$ as polynomials in $\{y_{ij}(k)\}$ is explicitly given in Section~\ref{section_main}. In Theorem~\ref{theo_main_finite} we establish that the set
$$H_{2,m}^{(2)}=S_{2,m}^{(2)}\sqcup \{\zeta(Y_i), \; 1\leq i\leq m; \;\;\; \eta_{\al}(Y_i,Y_j), \; 1\leq i<j\leq m,\; \al\in\FF\,\backslash\{0,1\}\}$$
is a minimal separating set for  $\mathcal{O}(\N_2^m)^{\GL_2}$ in case $\Char{\FF}=2$ and the set
$$H_{2,m}=S_{2,m}\sqcup \{\zeta(Y_i), \; 1\leq i\leq m; \;\;\; \eta_{\al}(Y_i,Y_j), \; 1\leq i<j\leq m,\; \al\in\FF\,\backslash\{0,1\}\}$$
is a minimal separating set for  $\mathcal{O}(\N_2^m)^{\GL_2}$  in case $\Char{\FF}>2$. 

In Theorem~\ref{theo_main_infinite} we describe a minimal separating set in case of an infinite field. Note that over an infinite field the functions $\zeta(Y_i)$ and $\eta_{\al}(Y_i,Y_j)$ do not belong to $\mathcal{O}(\N_2^m)$. Some corollaries are given in Section~\ref{section_cor}.

\subsection{Notations}\label{section_notations}

If for $\un{A},\un{B}\in M_n^m$ there exists $g\in \GL_n$ such that $g\cdot \un{A}=\un{B}$, then we write $\un{A}\sim \un{B}$ and say that $\un{A}$, $\un{B}$ are similar. We say that $\un{A}$ has no zeros if $A_i$  is non-zero for all $i$.   Denote by $\wz(\un{A})$ the result of elimination of all zero matrices from $\un{A}$. As an example, for $\un{A}=(0,B,0, C,D)$ with non-zero $B,C,D\in M_n$, we have $\wz(\un{A})=(B,C,D)$. For a permutation $\si\in \Sym_m$, we write $\un{A}_{\si}$ for $(A_{\si(1)},\ldots,A_{\si(m)})$.

Denote by $E$ the identity matrix and by $E_{ij}$ the matrix such that the $(i,j)^{\rm th}$ entry is equal to one and the rest of entries are zeros.   Denote by $\FF^{\times}$ the set of all non-zero elements of $\FF$.  For short, we write $\un{0}$ for $(0,\ldots,0)\in M_n^m$.  Given $A\in M_n$ and $r\geq 0$, we write $A^{(r)}$ for $(\underbrace{A,\ldots,A}_r)\in M_n^r$.


\section{Minimal separating invariants for $\mathcal{O}(W)^G$ over a finite field}\label{section_minimal}

In this section we assume that $\FF = \FF_q$,  $W\subset V$ is a $G$-invariant subset of $V$, and $\dim V=n$. Denote $\ga=\ga(q,\k) =\lceil\log_q(\k)\rceil$,  
where $\k$ stands for the number of $G$-orbits on $W$. For any vector $w=(w_1,\ldots,w_n)\in W$ consider the following polynomial function from $\mathcal{O}(W)$: 
$$f_w = (-1)^n \prod\limits_{\al\in\FF \setminus \{w_1 \}}(y_1 - \al) \;\;\cdots \prod\limits_{\al\in\FF\setminus \{w_n\}}(y_n - \al).$$
Note that $f_w(w)=(-1)^n \left(\prod\limits_{\al\in\FF^{\ast}} \al \right)^{\!\!n} =  1$ and for each $v\in W$ we have 
$$f_w(v)=\left\{
\begin{array}{cl}
1,& v=w \\
0,& \text{otherwise} \\
\end{array}
\right..$$
We present $W$ as a union of $G$-orbits: $W=U_1 \sqcup \cdots \sqcup U_{\k}$. For every $1\leq j\leq \k$ define
\begin{equation}\label{eq_fj}
f_j=\sum\limits_{u\in U_j} f_u.
\end{equation}
Then for each $v\in W$ we have
\begin{equation}\label{eq_f}
f_j(v)=\left\{
\begin{array}{cl}
1,& v\in U_j \\
0,& \text{otherwise} \\
\end{array}
\right..
\end{equation}
Therefore, $f_1,\ldots,f_{\k}$ belong to $\mathcal{O}(W)^G$ and they form a separating set for $\mathcal{O}(W)^G$. By the definition of $\ga$ we have that $q^{\ga-1}<\k\leq q^{\ga}$. Thus there exist~$\k$ pairwise different vectors in $\FF^\ga$, which we may put together in a matrix $(\al_{ij})$ over $\FF$ of size $(\ga \times \k)$.

\begin{theo}\label{theo1} 
\begin{enumerate}
\item[1.] Every separating set for $\mathcal{O}(W)^G$ contains at least $\ga(q,\k)$ elements. 

\item[2.] Let $(\al_{ij}) \in \FF^{\gamma \times \k}$ be a matrix whose columns are pairwise different and define $$h_i =  \al_{i1} f_1 + \cdots + \al_{i\k} f_{\k} \quad (1\leq i\leq \ga).$$
Then $\{h_1 \upto h_{\ga}\}$ is a minimal separating set for $\mathcal{O}(W)^G$ consisting of (non-homogeneous) elements of degree less or equal to $n(q-1)$.
\end{enumerate}
\end{theo}

\begin{proof} Consider some representatives of $G$-orbits on $W$: $u_1\in U_1 \upto u_{\k}\in U_{\k}$.

\medskip
\noindent{\bf 1.} Assume that $\{l_1 \upto l_r\}$ is a separating set for $\mathcal{O}(W)^G$. Then the column vectors \[ \begin{pmatrix}l_1(u_j)\\ \vdots \\ l_r(u_j) \end{pmatrix} \quad \text{ with } 1 \leq j \leq \k \] are pairwise different. Hence $\k \leq q^r$ and $\gamma \leq r$ follows from the definition of $\ga$.

\medskip
\noindent{\bf 2.} Applying formula~(\ref{eq_f}) we obtain that 
$$\begin{pmatrix}h_1(u_j) \\ \vdots \\ h_{\ga}(u_j) \end{pmatrix} = \begin{pmatrix}\al_{1j}\\ \vdots \\ \al_{\ga j} \end{pmatrix} \quad \text{ for all } 1\leq j\leq \k.$$ 
By the conditions of the theorem, these column vectors are pairwise different, hence $\{h_1,\ldots,h_{\ga}\}$ is a separating set. Minimality follows from part 1.
\end{proof}

\section{Classification of $\GL_2$-orbits on $\N_2^m$}\label{section_main1}

In this section we assume that $\FF$ is an arbitrary field. 

\begin{lemma}\label{lemma_1}
If $\un{A}\in\N_2^2$ has no zeros, then $\un{A}\sim (E_{12},\al E_{12})$ or $\un{A}\sim (E_{12}, \al E_{21})$ for some non-zero $\al\in\FF$.
\end{lemma}
\begin{proof} For each field $\FF$ we have that $A_1\sim E_{12}$. Therefore, without loss of generality we can assume that $A_1=E_{12}$. Denote $A_2=\matr{a_1}{a_2}{a_3}{-a_1}$ for some $a_1,a_2,a_3\in\FF$ and assume that $g=\matr{g_1}{g_2}{0}{g_1}$ lies in $\GL_2$. Then we have that $g\cdot A_1=A_1$ and
$$g\cdot A_2 = \matr{h_1}{h_2}{a_3}{-h_1},$$
for $h_1=a_1 + \frac{g_2}{g_1}a_3$ and $h_2=(a_2 g_1^2 - 2a_1 g_1 g_2 - a_3 g_2^2)/g_1^2$.

Assume that $a_3\neq0$. Then we take $g_1=1$ and $g_2=-\frac{a_1}{a_3}$ to obtain that $h_1=0$. Since $g\cdot A_2$ is nilpotent, we obtain $h_2=0$ and $g\cdot A_2=a_3 E_{21}$.

Assume that $a_3=0$. Since $A_2$ is nilpotent, we obtain $A_2=a_2 E_{12}$.
\end{proof}

For $\be,\ga,\de$ of $\FF$ denote 
$$D(\be,\ga,\de)=\matr{\be}{\ga}{\de}{-\be}.$$

\begin{remark}\label{remark_3}
Assume that $\un{A},\un{A}'\in \N_2^3$ satisfy $A_1=A'_1=E_{12}$, $A_2=\al E_{21}$, and $A'_2=\al' E_{21}$  for some $\al,\al'\in\FF^{\times}$. Assume that  $t_{ij}:=\tr(A_iA_j)-\tr(A'_iA'_j)=0$ for all $1\leq i<j\leq 3$ and $t_{123}:=\tr(A_1A_2A_3)-\tr(A'_1A'_2A'_3)=0$. Then $\un{A}=\un{A}'$. 
\end{remark}
\begin{proof} 
Denote  $A_3=D(\be,\ga,\de)$ and $A'_3=D(\be',\ga',\de')$ for some $\be,\ga,\de,\be',\ga',\de'$ from $\FF$.  Since $t_{12}=0$, we have $\al=\al'$. Since $t_{13}=0$, $t_{23}=0$, $t_{123}=0$, respectively, we obtain that $\de=\de'$, $\al \ga = \al \ga'$, $\al \be = \al \be'$, respectively.
\end{proof}

\begin{lemma}\label{lemma_3}
If $\un{A}\in\N_2^m$ has no zeros, then $\un{A}$ is similar to one and only one of the following elements:

\begin{enumerate}
\item[(a)] $(E_{12},\al_2 E_{12},\ldots,\al_m E_{12})$, where $\al_2,\ldots,\al_m\in \FF^{\times}$; 

\item[(b)] $(E_{12},\al_2 E_{12},\ldots,\al_r E_{12},\al_{r+1} E_{21}, D_1,\ldots,D_s)$, where $1\leq r\leq m-1$, $s=m-r-1$, $\al_2,\ldots,\al_{r+1}\in \FF^{\times}$, $D_i\in \N_2$ is non-zero for every $i$.
\end{enumerate}
\end{lemma}
\begin{proof} In case $m=1$ we have $A_1\sim E_{12}$ and the required is proven.

Assume $m\geq 2$. By Lemma~\ref{lemma_1} either $(A_1, A_2)\sim (E_{12}, \al_2 E_{12})$ or $(A_1,A_2)\sim (E_{12},\al_2 E_{21})$ for some $\al_2\in\FF^{\times}$. In the first case we have $A_1=A_2 / \al_2$ and applying Lemma~\ref{lemma_1} to the pair $(A_1,A_3)$ we can see that either $(A_1,A_2,A_3)\sim (E_{12},\al_2 E_{12},\al_3 E_{12})$ or $(A_1,A_2,A_3)\sim (E_{12},\al_2 E_{12},\al_3 E_{21})$ for some $\al_3\in\FF^{\times}$. Repeating this procedure we finally obtain that either case (a) or (b) holds.

To prove uniqueness consider some $\un{A},\un{A}'\in \N_2^m$ of type (a) or (b) and satisfying the condition  $\un{A}\sim \un{A}'$.   Note that $\un{A}$ has type (a) if and only if for every $2\leq i\leq m$ exists $\al_i\in\FF^{\times}$ such that $A_1=A_i/\al_i$. Since this property of $\un{A}$ is not changed by the action of $\GL_2$, then we have one of the following two cases. 

\medskip
\noindent{}{\bf (1)} $\un{A}$ and $\un{A}'$ have both type (a), i.e., 
$$\un{A}=(E_{12},\al_2 E_{12},\ldots,\al_m E_{12}) \;\; \text{ and } \;\;
\un{A}'=(E_{12},\al'_2 E_{12},\ldots,\al'_m E_{12}), $$
where $\al_2,\ldots,\al_m,\al'_2,\ldots,\al'_m\in \FF^{\times}$. Consider $g\in \GL_2$ such that $g\cdot \un{A}=\un{A}'$. Since $g\cdot A_i=A'_i$ for all $i$, we obtain that $g\cdot E_{12}=E_{12}$ and $\al_i = \al'_i$ for $2\leq i\leq m$. Therefore, $\un{A}=\un{A}'$. 

\medskip
\noindent{}{\bf (2)} $\un{A}$ and $\un{A}'$ have both type (b), i.e.,
$$
\begin{array}{ccc}
\un{A} &=&(E_{12},\; \al_2 E_{12},\ldots,\,\; \al_r E_{12},\; \al_{r+1} E_{21}, D_1,\ldots,D_s),\\
\un{A}' &=&(E_{12},\,\al'_2 E_{12},\ldots,\, \al'_{r'} E_{12},\al'_{r'+1} E_{21}, D'_1,\ldots,D'_{s'}),\\
\end{array}
$$
where $r,r'\geq 1$, $s,s'\geq0$, $\al_2,\ldots,\al_{r+1},\al'_2,\ldots,\al'_{r'+1}\in \FF^{\times}$, $D_i,D'_i\in \N_2$ is non-zero for every $i$.  By the definition of similarity we have that $t_{ij}:=\tr(A_iA_j)-\tr(A'_iA'_j)=0$ and $t_{ijk}:=\tr(A_iA_jA_k)-\tr(A'_iA'_jA'_k)=0$ for all $1\leq i,j,k\leq m$.
Since
$$\begin{array}{rcl}
\min\{2\leq v\leq m\,|\,  \tr(A_{v-1}A_{v})\neq 0\}&=&r+1,\\
\min\{2\leq v\leq m\,|\, \tr(A'_{v-1}A'_{v})\neq 0\}&=&r'+1,\\
\end{array}$$
we obtain that $r=r'$ and $s=s'$. Consider $1\leq i\leq s$. Applying Remark~\ref{remark_3} to the pair of triples $(A_1,A_{r+1},A_{r+i+1})$ and $(A'_1,A'_{r+1},A'_{r+i+1})$,  we obtain that $\al_{r+1}= \al'_{r+1}$ and $A_{r+i+1}=A'_{r+i+1}$. In case $r\geq2$ we use $t_{j,r+1}=0$ to obtain that $\al_j=\al'_j$ for all $2\leq j\leq r$. Therefore, $\un{A}=\un{A}'$.
\end{proof}

Lemma~\ref{lemma_3} implies the following description of $\GL_2$-orbits on $\N_2^m$.

\begin{theo}\label{theo_orb} 
Each $\GL_2$-orbit on $\N_2^m$ contains one and only one element of the following types: 

\begin{enumerate}
\item[(0)] $(0,\ldots,0)$; 

\item[(1)] $(\al_1 E_{12},\ldots,\al_m E_{12})$, where 
\begin{enumerate}
\item[$\bullet$] at least one element of the list $\al_1,\ldots,\al_m\in \FF$ is non-zero,

\item[$\bullet$] $\al_v=1$ for $v=\min\{1\leq i \leq m\,|\, \al_i\neq 0 \}$;
\end{enumerate}

\item[(2)] $(\al_1 E_{12},\ldots,\al_r E_{12},\al_{r+1} E_{21}, D_1,\ldots,D_s)$, where 
\begin{enumerate}
\item[$\bullet$] $1\leq r\leq m-1$, $s=m-r-1$, 

\item[$\bullet$] at least one element of the set $\{\al_1,\ldots,\al_{r}\}$ is non-zero and $\al_{r+1}\in\FF^{\times}$, 

\item[$\bullet$] $\al_v=1$ for $v=\min\{1\leq i \leq r\,|\, \al_i\neq 0 \}$, 

\item[$\bullet$]  $D_i\in \N_2$ for every $1\leq i\leq s$.
\end{enumerate}
\end{enumerate}
\end{theo}
\bigskip

For every integer $k$ denote 
$$S_{k}(q)=
\left\{
\begin{array}{rl}
q^k+q^{k-1}+\cdots+1 = \frac{q^{k+1}-1}{q-1}, & \text{ if } k\geq0 \\
0 , & \text{ if } k<0 \\
\end{array}
\right..
$$

\begin{cor}\label{cor_cor}
Assume $\FF=\FF_q$ and $m\geq1$. Then
\begin{enumerate}
\item[(a)] the number of $\GL_2$-orbits on $\N_2^m$ is equal to 
$$\k= 1 + \frac{(q^m-1)(q^{m-1}+q)}{q^2-1};$$ 

\item[(b)] $\k$ has the following presentation as a polynomial in $q$ with integer non-negative coefficients:
$$ \k =\k(q) = \left\{
\begin{array}{rl}
1 + S_{m-1}(q) + (q^{m-1} - 1) S_{\frac{m-2}{2}}(q^2), & \text{ if } m \text{ is even} \\
1 + S_{m-1}(q) + (q^{m} - 1) S_{\frac{m-3}{2}}(q^2), & \text{ if } m \text{ is odd} \\
\end{array}
\right.;$$

\item[(c)] the least possible number of elements for a separating set for $\mathcal{O}(\N_2^m)^{\GL_2}$ is 
$$\ga=
\left\{
\begin{array}{rl}
1, & \text{ if }m=1  \\
3, &\text{ if }m=q=2  \\
2m-2, &\text{ otherwise } \\
\end{array}
\right..$$
\end{enumerate}
\end{cor}
\begin{proof}
\noindent{\bf (a)} Denote by $\k_0(m)$, $\k_1(m)$, $\k_2(m)$, respectively, the number of $\GL_2$-orbits on $\N_2^m$ of type (0), (1), (2), respectively (see  Theorem~\ref{theo_orb}). Then $\k_0(m)=1$, $\k_1(m)=S_{m-1}(q)$, and 
$$\k_2(m) = \sum_{r=1}^{m-1}\k_1(r) (q-1) |\N_2|^{m-r-1} = \sum_{r=1}^{m-1} (q^r-1) q^{2m-2r-2},$$
since we apply the equality $|\N_n|=q^{n(n-1)}$, which was proven by Fine and Herstein~\cite{Fine_Herstein_1958}. Hence,
$$\k_2(m) = q^{2m-2} \left(\sum_{r=0}^{m-1} q^{-r} - \sum_{r=0}^{m-1} q^{-2r}\right)= 
q^{2m-2} \left(S_{m-1}(q^{-1}) - S_{m-1}(q^{-2})\right),$$
and we obtain
$$
\k_2(m) = \frac{(q^m-1)(q^{m-1}-1)}{q^2-1}.
$$
Therefore,
\begin{eq}\label{eq_k}
\k= \k_0(m)+\k_1(m)+\k_2(m) = 1 + S_{m-1}(q)+\frac{(q^m-1)(q^{m-1}-1)}{q^2-1},
\end{eq}%
and the claim of part (a) easily follows.
\medskip

\noindent{\bf (b)} If $m=1$, then $\k=2$ and part (b) holds. Assume $m\geq2$. Considering the case of even $m$ and the case of odd $m$, applying formula~(\ref{eq_k}), we prove the claim of part (b).

\medskip

\noindent{\bf (c)} By Theorem~\ref{theo1}, the least possible number of elements for a separating set for  $\mathcal{O}(\N_2^m)^{\GL_2}$ is $\ga=\lceil\log_q(\k)\rceil$. 

Assume $m=1$. Then $\k=2$ and $\ga=1$. 

Assume $m=q=2$. Then $\k=5$ and $\ga=3$.

Assume $m\geq 2$ and $m,q$ are not simultaneously equal to $2$. By part (a), the claim that $\ga=2m-2$ is equivalent to inequalities 
$$q^{2m-3}< 1 + \frac{(q^m-1)(q^{m-1}+q)}{q^2-1} \leq q^{2m-2}.$$%
The left inequality follows from
$$q^{2m-3}< \frac{(q^m-1)(q^{m-1}+q)}{q^2-1}.$$%
The right inequality is equivalent to
\begin{eq}\label{eq_k_c}
(q^{2m}-q^{2m-1}-q^{2m-2} - q^{m+1}) + (q^{m-1}- q^2) + q  + 1\geq 0.
\end{eq}%

In case $m=2$ and $q\geq3$, inequality~(\ref{eq_k_c}) is equivalent to $q^2(q^2-2q-2)+2q+1\geq0$, which holds since $q^2-2q-2\geq0$ for $q\geq3$.

In case $m\geq3$ we have $q^{2m}=q^{2m-2}q^2\geq q^{2m-2}q + q^{2m-2} + q^{2m-2}\geq q^{2m-1}+q^{2m-2} + q^{m+1}$, since $q^2\geq q+2$ for $q\geq2$; inequality~(\ref{eq_k_c}) follows. 
\end{proof}

Corollary~\ref{cor_cor} implies the following remark:

\begin{remark} Assume $\FF=\FF_q$. Then
\begin{enumerate}
\item[$\bullet$] for $m=1$ we have $\k=2$; 

\item[$\bullet$] for $m=2$ we have $\k=2q+1$; 

\item[$\bullet$] for $m=3$ we have $\k=q^3 + q^2 + q + 1$; 

\item[$\bullet$] for $m=4$ we have $\k=q^5 + 2q^3 + q + 1$; 

\item[$\bullet$] for $m=5$ we have $\k= q^7 + q^5 + q^4 + q^3 + q + 1$. 
\end{enumerate}
\end{remark}
\medskip

Corollary~\ref{cor_cor} implies that the following conjecture holds for $n=2$ and $m\geq1$. 

\begin{conj}\label{conj}
Assume $\FF=\FF_q$. If we fix $n\geq2$ and $m\geq1$, then the number of $\GL_n$-orbits on $\N_n^m$ is a polynomial in $q$ with integer coefficients. 
\end{conj}
\bigskip

It was proven by Hua~\cite{Hua_2021} that  the number of $\GL_n$-orbits on $\N_n^m$ is a polynomial in $q$ with rational coefficients. Hua also conjectured that the number of absolutely indecomposable $\GL_n$-orbits on $\N_n^m$ is a polynomial in $q$ with non-negative integer coefficients (see Conjecture 4.1 of~\cite{Hua_2021}). This conjecture was shown to be true for $m=2$ and $1\leq n\leq 6$.

\section{Separating set for $\mathcal{O}(\N_2^m)^{\GL_2}$}\label{section_main}

If $\FF=\FF_q$ and $A=\matr{a_1}{a_2}{a_3}{a_4}$ lies in $\N_2$, then we have that $\zeta(A)=a_1^{q-1} - a_2^{q-1} - a_3^{q-1} + 1$,  since $a_1^2 =-a_2a_3$. In particular, 
\begin{eq}\label{eq_zeta2}
\zeta(Y_i)=y_{11}\!(i)^{q-1} - y_{12}\!(i)^{q-1} - y_{21}\!(i)^{q-1} + 1.
\end{eq}
Similarly, we can see that

\begin{eq}\label{eq_eta2}
\eta_{\al}(Y_i,Y_j)=\prod\limits_{u,v\in\{1,2\}}((\al\, y_{uv}(i)-y_{uv}(j))^{q-1}-1).
\end{eq}

%

\begin{remark}\label{remark_basic}
Let $H$ be one of the following sets: $S_{2,m}$, $S_{2,m}^{(2)}$, $H_{2,m}$, $H_{2,m}^{(2)}$. Assume that $\un{A},\un{B}\in\N_2^m$ are not separated by $H$. Then  
 \begin{enumerate}
 \item[(a)]  for any $\si\in\Sym_m$  we have that $\un{A}_{\si},\un{B}_{\si}$ are not separated by $H$;

 \item[(b)] for any $\un{A}',\un{B}'\in \N_n^m$ with $\un{A}\sim\un{A}'$ and $\un{B}\sim\un{B}'$ we have that $\un{A}',\un{B}'$ are not separated by $H$.
\end{enumerate}
\end{remark}
\begin{proof} Consider some $A_1,A_2,A_3$ from $\N_2$. It is easy to see that $\eta_{\al}(A_1,A_2)=\eta_{\al^{-1}}(A_2,A_1)$ for all $\al\in\FF^{\times}$. Moreover, the Cayley–-Hamilton theorem implies that 
\begin{eq}\label{eq_trABC}
\tr(A_1A_2A_3)=-\tr(A_1A_3A_2).
\end{eq}
Thus part (a) is proven. Part (b) is trivial. 
\end{proof}

\begin{lemma}\label{lemma_4} Assume that $\un{A},\un{B}\in\N_2^m$ are not separated by $S_{2,m}$. Then there exists a permutation $\si\in\Sym_m$ such that one of the next cases holds:
\begin{enumerate}
\item[(a)] 
$$\begin{array}{ccl}
\un{A}_{\si} & \sim & (0^{(k+l)},\al_1 E_{12},\ldots,\al_{r+s} E_{12}),\\ \un{B}_{\si} & \sim & (0^{(k)},\be_1 E_{12},\ldots,\be_l E_{12}, 0^{(r)}, \be_{l+1} E_{12},\ldots,\be_{l+s} E_{12}),\\ 
\end{array}
$$
where $k,l,r,s\geq0$, $k+l+r+s=m$, $\al_1,\ldots,\al_{r+s},\be_1,\ldots,\be_{l+s}\in \FF^{\times}$; moreover, $\al_1=1$ in case $\un{A}\neq \un{0}$ and $\be_1=1$ in case $\un{B}\neq \un{0}$; 

\item[(b)] $\un{A}_{\si}\sim (0^{(k)}, E_{12},\al_2 E_{12}, \ldots, \al_r E_{12}, \al_{r+1} E_{21}, D_1,\ldots,D_s)$ and $\un{B}\sim\un{A}$, where  $r\geq 1$, $k,s\geq0$, $k+r+s+1=m$, $\al_2,\ldots,\al_{r+1}\in\FF^{\times}$, and $D_i\in\N_2$ is non-zero for every $i$.
\end{enumerate}
 \end{lemma}
\begin{proof} We use Remark~\ref{remark_basic} without reference to it. Applying Lemma~\ref{lemma_3} to $\un{A}$ together with $\Sym_m$-action on $\un{A}$ we can see that one of the next two cases holds.

\medskip
\noindent{}{\bf (1)}  $\un{A}_{\si}\sim (0^{(k)},\al_1 E_{12},\ldots,\al_{r} E_{12})$, where $k,r\geq0$, $k+r=m$, and $\al_1,\ldots,\al_r\in\FF^{\times}$ with $\al_1=1$ in case $\un{A}\neq \un{0}$. Applying Lemma~\ref{lemma_3} to $\un{B}_{\si}$ together with $\Sym_m$-action on $(\un{A}_{\si},\un{B}_{\si})$ we can see that  one of the next two cases holds:
\begin{enumerate}
\item[(1a)] there is $\tau\in \Sym_m$ such that $(\un{A}_{\tau},\un{B}_{\tau})$ satisfies case (a) maybe without condition that $\be_1=1$ in case $\un{B}\neq\un{0}$. In the latter case we act by $\matr{\be_1^{-1}}{0}{0}{1}\in\GL_2$ on $\un{B}_{\tau}$ to obtain case (a).  

\item[(2a)] $\un{A}_{\si}$ is the same as above and $\un{B}_{\si}\sim \un{C}$, where for some $i<j$ we have $C_i=E_{12}$ and $C_j=\al E_{21}$ for  $\al\in\FF^{\times}$. In this case we have $\tr(A_{\si(i)}A_{\si(j)})=\tr(C_iC_j)$; hence, $\al=0$, a contradiction. 
\end{enumerate}

\medskip
\noindent{}{\bf (2)}   $\un{A}_{\si}\sim (0^{(k)},\al_1 E_{12}, \ldots, \al_r E_{12}, \al_{r+1} E_{21}, D_1,\ldots,D_s)$, where  $r\geq 1$, $k,s\geq0$, $\al_1=1$, $\al_2,\ldots,\al_{r+1}\in\FF^{\times}$, and $D_i\in\N_2$ is non-zero for every $i$.  Applying Lemma~\ref{lemma_3} to $\un{B}_{\si}$  we can see that one of the next two cases holds:
\begin{enumerate}
\item[(2a)] $wz(\un{B}_{\si})\sim(\be_1 E_{12},\ldots,\be_l E_{12})$ for some $l\geq0$ and $\be_1,\ldots,\be_l\in\FF^{\times}$, where $\be_1=1$ in case $\un{B}\neq\un{0}$. Then $B_iB_j=0$ for all $1\leq i,j\leq m$. On the other hand, $\tr(A_{\si(k+r)}A_{\si(k+r+1)})=\al_r \al_{r+1}$ is non-zero; a contradiction. 

\item[(2b)] 
$wz(\un{B}_{\si})\sim (\be_1 E_{12},\ldots,\be_{r'} E_{12},\be_{r'+1} E_{21}, D'_1,\ldots,D'_{s'})$, where $r'\geq 1$, $s'\geq0$, $\be_1=1$, $\be_2,\ldots,\be_{r'+1}\in \FF^{\times}$, $D'_i\in \N_2$ is non-zero for every $1\leq i\leq s'$. Therefore, 
$$\un{B}_{\si}\sim(\ldots,\be_{r'}E_{12},0,\ldots,0,\be_{r'+1} E_{21}, \ldots),$$
\vspace{-0.5cm}
$$\qquad \; \widehat{i} \qquad \qquad\quad\;\; \widehat{j} $$
where $\be_{r'}E_{12}$ and $\be_{r'+1} E_{21}$, respectively, stands in position $i$ and $j$, respectively, for some $1\leq i<j\leq m$. Since
$$\begin{array}{rcl}
\min\{1\leq v\leq m\,|\,\exists\, 1\leq u <v \text{ such that } 
\tr(A_{\si(u)}A_{\si(v)})\neq 0\}&=&k+r+1,\\
\min\{1\leq v\leq m\,|\,\exists\,  1\leq u <v \text{ such that } 
\tr(B_{\si(u)}B_{\si(v)})\neq 0\}&=&j,\\
\end{array}
$$
the conditions of the lemma imply that $j=k+r+1$. For $1\leq u\leq k+r$ we have 
\begin{eq}\label{eq_2b}
\tr(B_{\si(u)}B_{\si(k+r+1)})=\tr(A_{\si(u)}A_{\si(k+r+1)})=
\left\{
\begin{array}{rl}
\al_{u-k} \al_{r+1}, & \text{if } u>k\\
0, & \text{if }  u\leq k\\
\end{array}
\right..
\end{eq}%
Therefore,  $\un{B}_{\si}\sim (0^{(k)},\be_1 E_{12},\ldots,\be_r E_{12},\be_{r+1} E_{21}, D''_1,\ldots,D''_s)$, where some of the matrices $D''_1,\ldots,D''_s\in\N_2$ can be zero. In particular, $r=r'$. Since $\al_1=\be_1=1$, equality~\Ref{eq_2b} implies that $\al_{r+1}=\be_{r+1}$; therefore, $\al_i=\be_i$ for all $2\leq i\leq r$. Applying Remark~\ref{remark_3} to the pair of triples $(A_{\si(k+r)},A_{\si(k+r+1)},A_{\si(k+r+1+u)})$ and $(B_{\si(k+r)},B_{\si(k+r+1)},B_{\si(k+r+1+u)})$, where $1\leq u\leq s$, we obtain that $D_u=D'_u$. Therefore, $\un{A}_{\si}\sim 
\un{B}_{\si}$ and case (b) of the formulation of this lemma holds. 
\end{enumerate}
\end{proof}

\begin{remark}\label{remark_4.3}
For $\un{A}\in\N_2^2$ we have 
$$
\eta_{1}(A_1,A_2) = \left\{
\begin{array}{cl}
1 - \sum\limits_{\al\in \FF\,\backslash \{0,1\}} \eta_{\al}(A_1,A_2), & 
\text{if } A_1\neq0,\; A_2\neq0,\; \tr(A_1A_2)=0 \\
1  & \text{if } A_1=A_2=0 \\
0  & \text{if } A_1\neq 0, A_2=0 \;\text{ or }\; A_1= 0, A_2\neq 0  \\
0  & \text{if } \tr(A_1A_2)\neq 0 \\
\end{array}
\right..
$$
\end{remark}
\begin{proof} Since $\eta_{\al}(A_1,A_2)$ and $\tr(A_1A_2)$ are constants on the $\GL_2$-orbits on $\N_2^2$, by Lemma~\ref{lemma_1} we can assume that on  $\un{A}$ lies in the following list: $(0,0)$, $(0,E_{12})$, $(E_{12},0)$, $(E_{12},\ga E_{12})$, $(E_{12},\ga E_{21})$, where $\ga\in\FF^{\times}$. The required follows from case by case consideration. 
\end{proof}

\begin{lemma}\label{lemma_4.4}
Assume that $\Char{\FF}=2$ and $\un{A},\un{B}\in\N_2^3$ are not separated by $S_3^{(2)}=\{\tr(Y_1Y_2),\, \tr(Y_1Y_3),\, \tr(Y_2Y_3)\}$. Then $\tr(A_1 A_2 A_3)=\tr(B_1 B_2 B_3)$.
\end{lemma}
\begin{proof} 
If $A_i=B_j=0$ for some $1\leq i,j\leq 3$, then $\tr(A_1 A_2 A_3)=\tr(B_1 B_2 B_3)=0$. Therefore, without loss of generality we can assume that $A_i$ is non-zero for all $i$.

\medskip
\noindent{\bf 1.} Assume that $B_j=0$ for some $j$. By formula~\Ref{eq_trABC}, without loss of generality we can assume that $B_1=0$. Applying Lemma~\ref{lemma_3} to $\un{A}$ we obtain that one of the following cases holds:
$$({\rm a})\; \un{A}\sim(E_{12},\al_2 E_{12}, \al_3 E_{12}),\qquad  
({\rm b})\;    \un{A}\sim(E_{12},\al_2 E_{21}, D),\qquad 
({\rm c})\;    \un{A}\sim(E_{12},\al_2 E_{12}, \al_3 E_{21}),$$
where $\al_2,\al_3\in\FF^{\times}$ and $D\in\N_2$ is non-zero. In cases (a) and (c) we have $\tr(A_1A_2A_3)=0$ and the required holds. In case (b) we have $\al_2=\tr(A_1A_2)=\tr(B_1B_2)=0$; a contradiction.

\medskip
\noindent{\bf 2.} Assume that $B_j$ is non-zero for all $j$.  Applying Lemma~\ref{lemma_3} to $\un{A}$  we obtain that one of the above cases (a), (b), (c) holds for $\un{A}$.  Applying Lemma~\ref{lemma_3} to $\un{B}$  we obtain that one of the following cases holds:
$$({\rm a'})\; \un{B}\sim(E_{12},\be_2 E_{12}, \be_3 E_{12}),\qquad  
({\rm b'})\;    \un{B}\sim(E_{12},\be_2 E_{21}, D'),\qquad 
({\rm c'})\;    \un{B}\sim(E_{12},\be_2 E_{12}, \be_3 E_{21}),$$
where $\be_2,\be_3\in\FF^{\times}$ and $D'\in\N_2$ is non-zero. Since $\tr(A_1 A_2)=\tr(B_1 B_2)$ and $\tr(A_1 A_3)=\tr(B_1 B_3)$, we have one of the following three cases:
\begin{enumerate}
\item[$\bullet$] cases (a) and (${\rm a'}$) hold. Then $\tr(A_1A_2A_3)=\tr(B_1B_2B_3)=0$. 

\item[$\bullet$] cases (b) and (${\rm b'}$) hold. Since $\tr(A_1A_2)=\tr(B_1B_2)$, we obtain that $\al_2=\be_2$. Denote $D=D(a_1,a_2,a_3)$ and  $D'=D(b_1,b_2,b_3)$. The equalities $\tr(A_1A_3)=\tr(B_1B_3)$ and $\tr(A_2A_3)=\tr(B_2B_3)$, respectively, imply that $a_3=b_3$ and $a_2=b_2$, respectively. It follows from equalities $\det(D)=\det(D')=0$ that $a_1=b_1$, since $\Char\FF=2$. Therefore, $\un{A}\sim\un{B}$ and the required follows.

\item[$\bullet$] cases (c) and (${\rm c'}$) hold. Then $\tr(A_1A_2A_3)=\tr(B_1B_2B_3)=0$.

\end{enumerate}
\end{proof}

\begin{theo}\label{theo_main_finite}
Assume that $\FF=\FF_q$ is finite. Then 
\begin{enumerate}
\item[$\bullet$] $H_{2,m}^{(2)}$, in case $\Char{\FF}=2$;

\item[$\bullet$] $H_{2,m}$, in case $\Char{\FF}>2$; 
\end{enumerate}
is a minimal separating set for the algebra of invariant polynomial functions on nilpotent matrices  $\mathcal{O}(\N_2^m)^{\GL_2}$ for all $m>0$. 
\end{theo}
\begin{proof} Denote by $H$ the set $H_{2,m}^{(2)}$ or $H_{2,m}$, respectively, in case $\Char{\FF}=2$ or $\Char{\FF}>2$, respectively. 
Assume that $\un{A},\un{B}\in \N_2^m$  are not separated by $H$. To prove that $H$ is separating it is enough to show that $\un{A}\sim\un{B}$. By Remarks~\ref{remark_basic}, \ref{remark_4.3} and Lemma~\ref{lemma_4.4} we obtain in both cases that  $\un{A},\un{B}$ are not separated by 
$$H'_{2,m}=H_{2,m}\cup \{\eta_{\al}(Y_i,Y_j), \; 1\leq i,j\leq m,\; \al\in\FF^{\times}\}.$$

Applying Lemma~\ref{lemma_4} together with the fact $A_i=0$ if and only if $B_i=0$ ($1\leq i
\leq m$), we obtain that one of the next two cases holds:
\begin{enumerate}
\item[$\bullet$] $\un{A}_{\si}  \sim  (0^{(k)}, E_{12},\al_{k+2} E_{12},\ldots,\al_m E_{12})$, $\un{B}_{\si}  \sim  (0^{(k)}, E_{12},\be_{k+2} E_{12},\ldots,\be_m E_{12})$, 
where $\si\in \Sym_m$, $0\leq k<m$, $\al_{k+2},\ldots,\al_m,\be_{k+2},\ldots,\be_m\in \FF^{\times}$; 

\item[$\bullet$] $\un{A}\sim\un{B}$.
\end{enumerate}
Assume that the first case holds. Since for every $k+2\leq i\leq m$ and every $\al\in\FF^{\times}$ we have
$$A_{\si(k+1)}-\al A_{\si(i)} = 0 \; \text{ if and only if } \; B_{\si(k+1)}-\al B_{\si(i)} = 0,$$
then $\al_i=\be_i$. Therefore, $\un{A}_{\si}\sim\un{B}_{\si}$ and $H$ is separating.

Claims 1--4 (see below) show that $H$ is a {\it minimal} separating set  for all $m>0$. 

\medskip
\noindent{}{\it Claim 1.} The set $H_{2,1}\backslash \{\zeta(Y_1)\}=\emptyset$ is not separating for $\mathcal{O}(\N_2)^{\GL_2}$.
\smallskip

To prove this claim it is enough to consider $\un{A}=(0)$ and $\un{B}=(E_{12})$.

\medskip
\noindent{}{\it Claim 2.} Given $\be\in \FF\,\backslash\{0,1\}$, the set  $H_{2,2}\backslash \{\eta_{\be}(Y_1, Y_2)\}$ is not separating for $\mathcal{O}(\N_2^2)^{\GL_2}$.
\smallskip

To prove this claim consider $\un{A}=(E_{12},E_{12})$ and $\un{B}=(E_{12},\be E_{12})$ from $\N_2^2$. Then $\tr(A_1A_2)=\tr(B_1B_2)=0$ and $\eta_{\al}(A_1,A_2)=\eta_{\al}(B_1,B_2)=0$ for all $\al\in \FF\,\backslash\{0,1,\be\}$, but $\eta_{\be}(A_1,A_2)\neq\eta_{\be}(B_1,B_2)$. 

\medskip
\noindent{}{\it Claim 3.}  The set  $H_{2,2}\backslash \{\tr(Y_1 Y_2)\}$ is not separating for $\mathcal{O}(\N_2^2)^{\GL_2}$.
\smallskip

To prove this claim consider $\un{A}=(E_{12},E_{12})$ and $\un{B}=(E_{12},E_{21})$ from $\N_2^2$. Then $\eta_{\al}(A_1,A_2) =\eta_{\al}(B_1,B_2) = 0$ for all $\al\in \FF\,\backslash\{0,1\}$, but $\tr(A_1 A_2)\neq \tr(B_1 B_2)$.

\medskip
\noindent{}{\it Claim 4.} Let $m=3$ and $\Char{\FF}\neq2$. Then $H_{2,3}\backslash \{\tr(Y_1 Y_2 Y_3)\}$ is not separating for $\mathcal{O}(\N_2^3)^{\GL_2}$.
\smallskip

To prove this claim we consider $\un{A}=(E_{12},E_{21}, A_3)$ and $\un{B}=(E_{12},E_{21},B_3)$ from $\N_2^3$, where 
$$A_3=\matr{1}{1}{-1}{-1} \text{ and } B_3=\matr{-1}{1}{-1}{1}.$$
Thus $\eta_{\al}(A_i,A_j) =\eta_{\al}(B_i,B_j) = 0$ for all $1\leq i<j\leq 3$ and $\al\in \FF\,\backslash\{0,1\}$. Moreover, $\tr(A_1 A_2)=\tr(B_1 B_2)=1$, $\tr(A_1 A_3)=\tr(B_1 B_3)=-1$, $\tr(A_2 A_3)=\tr(B_2 B_3)=1$.  On the other hand, $\tr(A_1 A_2 A_3)=1$ and $\tr(B_1 B_2 B_3)=-1$ are different.
\end{proof}

\section{The case of infinite field}\label{section_main_appl}

In this section we assume that $\FF$ is infinite. We say that a $G$-invariant subset $W\subset V$ is a $G$-invariant  {\it cone} if the conditions  $w\in W$ and $\al\in\FF^{\times}$ imply that $\al w\in W$. 

\begin{remark}\label{remark_cone_degree} If $W\subset V$ is a $G$-invariant cone, then the algebra $\mathcal{O}(W)^G$ has $\NN$-grading by degrees.
\end{remark}
\smallskip

Note that $\N_n^m$ is a {\it multi-cone}, i.e., if $\un{A}\in\N_n^m$ and $\al_1,\ldots,\al_m\in\FF^{\times}$, then $(\al_1 A_1,\ldots,\al_m A_m)\in\N_n^m$. Thus the following extension of Remark~\ref{remark_cone_degree} to the case of multidegree holds. 

\begin{remark}\label{remark_cone} The algebra $\mathcal{O}(\N_n^m)^{\GL_n}$ has $\NN^m$-grading by multidegrees, where the multidegree of  $f\in \mathcal{O}(\N_n^m)\simeq \FF[M_n^m] / I(\N_n^m)$ is defined as the minimal multidegree (with respect to some fixed lexicographical order) of a polynomial $h\in\FF[M_n^m]$  with $f=h+I(\N_n^m)$.
\end{remark}
\bigskip

\begin{lemma}\label{lemma_21}
If $f\in \mathcal{O}(\N_2^m)^{\GL_2}$ is $\NN^m$-homogeneous and $f\not\in \FF$, then $f(\al_1 E_{12},\ldots,\al_m E_{12})=0$ for all $\al_1,\ldots,\al_m\in \FF$.
\end{lemma}
\begin{proof} Denote the multidegree of $f$ by $\un{\de}=(\de_1,\ldots,\de_m)$, where $\de_!+\cdots+\de_m>0$. Then 
$$f=\be\, y_{12}(1)^{\de_1}\cdots y_{12}(m)^{\de_m} + \sum_{h\in \Omega}\be_h h,$$
for some $\be,\be_h\in\FF$, where $\Omega$ is the set of all monomials in $\{y_{ij}(k)\,|\,i,j\in\{1,2\},\, 1\leq k\leq m\}$ different from $y_{12}(1)^{d_1}\cdots y_{12}(m)^{d_m}$ for all $d_1,\ldots,d_m\geq0$.

For $g=\matr{\ga}{0}{0}{1}\in \GL_2$ with $\ga\in\FF^{\times}$ we have
$$ \begin{array}{ccccc}
\be & = &f(E_{12},\ldots,E_{12}) &= & f(g\cdot(E_{12},\ldots,E_{12})) \\
& = & f(\ga E_{12},\ldots,\ga E_{12}) &= & \be\,\ga^{\de_1 +\cdots + \de_m}\\
\end{array}$$
for all $\ga\in\FF^{\times}$. Hence $\be=0$, since $\FF$ is infinite. Therefore, $f(\al_1 E_{12},\ldots,\al_m E_{12})=\be=0$. 
\end{proof}

\begin{theo}\label{theo_main_infinite}
Assume that $\FF$ is infinite.  Then 
\begin{enumerate}
\item[$\bullet$] $S_{2,m}^{(2)}$, in case $\Char{\FF}=2$;

\item[$\bullet$] $S_{2,m}$, in case $\Char{\FF}>2$; 
\end{enumerate}
is a minimal separating set for the algebra of invariant polynomial functions on nilpotent matrices  $\mathcal{O}(\N_2^m)^{\GL_2}$ for all $m>0$. 
\end{theo}
\begin{proof} Denote by $S$ the set $S_{2,m}^{(2)}$ or $S_{2,m}$, respectively, in case $\Char{\FF}=2$ or $\Char{\FF}>2$, respectively. 
Assume that $\un{A},\un{B}\in \N_2^m$ are not separated by $S$ and $f\in \mathcal{O}(\N_2^m)^{\GL_2}$. To prove that $S$ is separating we have to show that $f(\un{A})=f(\un{B})$.  By Remark~\ref{remark_cone}, 
without loss of generality we can assume that $f$ is $\NN^m$-homogeneous of multidegree $\un{\de}$. Moreover, without loss of generality we can assume that $f(\un{A})\neq0$. By Lemmas~\ref{lemma_4}, \ref{lemma_4.4} one of the next two cases holds:
\begin{enumerate}
\item[$\bullet$] $$\begin{array}{ccl}
\un{A}_{\si} & \sim & (0^{(k+l)},\al_1 E_{12},\ldots,\al_{r+s} E_{12}),\\ \un{B}_{\si} & \sim & (0^{(k)},\be_1 E_{12},\ldots,\be_l E_{12}, 0^{(r)}, \be_{l+1} E_{12},\ldots,\be_{l+s} E_{12}),\\ 
\end{array}
$$
where $\si\in \Sym_m$, $k,l,r,s\geq0$, $\al_1,\ldots,\al_{r+s},\be_1,\ldots,\be_{l+s}\in \FF^{\times}$; moreover, $\al_1=1$ in case $\un{A}\neq 0$ and $\be_1=1$ in case $\un{B}\neq 0$; 

\item[$\bullet$] $\un{A}\sim\un{B}$.
\end{enumerate}
In the second case we have  $f(\un{A})=f(\un{B})$. Assume that the first case holds. 

Define $f_{\si^{-1}}$ as the result of substitutions $y_{ij}(t)\to y_{ij}(\si^{-1}(t))$ in $f$ for all $1\leq t\leq m$, $1\leq i,j\leq 2$. Then $f_{\si^{-1}}\in \mathcal{O}(\N_2^m)^{\GL_2}$ and $f_{\si^{-1}}(\un{C}_{\si})=f(\un{C})$ for all $\un{C}\in \N_2^m$. Therefore, to show that $f(\un{A})=f(\un{B})$ it is enough to show that $f_{\si^{-1}}(\un{A}_{\si})=f_{\si^{-1}}(\un{B}_{\si})$. Note that $\un{A}_{\si}$, $\un{B}_{\si}$ are not separated by $S$ by Remark~\ref{remark_basic}. Therefore, without loss of generality we can assume that $\si\in\Sym_m$ is the trivial permutation.

If $\de_i\neq0$ for some $1\leq i\leq k+l$, then $f(\un{A})=0$; a contradiction. Otherwise, $\de_1=\cdots=\de_{k+l}=0$. Hence $f$ is a polynomial in $\{y_{ij}(k+l+1),\ldots, y_{ij}(m)\,|\,i,j\in\{1,2\}\}$. Therefore,
$f(E_{12}^{(k+l)},\al_1 E_{12},\ldots,\al_{r+s} E_{12})=f(\un{A})$ is non-zero; a contradiction to Lemma~\ref{lemma_21}. Thus $S$ is separating.

Taking $\un{A},\un{B}$ from Claims 3, 4 from the proof of Theorem~\ref{theo_main_finite} we obtain that $S$ is a {\it minimal} separating set for all $m>0$. 
\end{proof}

\section{Corollaries}\label{section_cor}

As in Section~\ref{section_inv}, assume that $V$ is an $n$-dimensional vector space over $\FF$,  $G$ is a subgroup of $\GL(V)$, and $W\subset V$ is a $G$-invariant subset of $V$. We say that an $m_0$-tuple $\un{j}\in\NN^{m_0}$ is {\it $m$-admissible} if $1\leq j_1<\cdots < j_{m_0}\leq m$. For any $m$-admissible $\un{j}\in\NN^{m_0}$ and $f\in \mathcal{O}(W^{m_0})^{G}$ we define the invariant polynomial function $f^{(\un{j})}\in \mathcal{O}(W^{m})^{G}$ as the result of the following substitution of variables in $f$: 
\[ y_{1,i} \to y_{j_1,i},\, \ldots, \, y_{m_0,i} \to y_{j_{m_0},i} \quad \text{ (for all } 1 \leq i \leq n).\]
Given a set $S \subset \mathcal{O}(W^{m_0})^{G}$, we define its {\it expansion} $S^{[m]} \subset \mathcal{O}(W^m)^{G}$ by 
\begin{eq}\label{eq_expan}
S^{[m]} = \{ f^{(\un{j})} \,|\, f \in S \text{ and } \un{j}\in\NN^{m_0} \text{ is $m$-admissible}\}.
\end{eq}

\begin{remark} (Cf. \cite[Remark~1.3]{Domokos_2007})\label{remark_expan}
Assume that $S_{1}$ and $S_{2}$ are separating sets for $\mathcal{O}(V^{m_0})^{G}$ and assume that $m>m_0$. Then $S_1^{[m]}$ is separating for $\mathcal{O}(V^{m})^{G}$ if and only if $S_2^{[m]}$ is separating for $\mathcal{O}(V^{m})^{G}$.
\end{remark}
\medskip

Denote by $\sisep(\mathcal{O}(W),G)$ the minimal number $m_0$ such that the expansion of some separating set $S$ for $\mathcal{O}(W^{m_0})^{G}$ produces a separating set for $\mathcal{O}(W^{m})^{G}$ for all $m \geq m_0$. It immediately follows from the main result of~\cite{Lopatin_Reimers_1} that $\sisep(\mathcal{O}(V),\Sym_n) \leq \lfloor \frac{n}{2} \rfloor + 1$ over an arbitrary field $\FF$, where $\Sym_n$ acts on $V$ by the permutation of the coordinates. Moreover, $\sisep(\mathcal{O}(V),\Sym_n)\leq \lfloor  \log_2(n)\rfloor + 1$ in case $\FF=\FF_2$ (see Corollary 4.12 of~\cite{Kemper_Lopatin_Reimers_2}).

\begin{cor}\label{cor_main_finite_be}
Assume that $\FF=\FF_q$ is finite and $m\geq2$. Then 
\begin{enumerate}
\item[$\bullet$] $\besep(\mathcal{O}(\N_2^m)^{\GL_2})\leq 2$, in case $q=2$;

\item[$\bullet$] $\besep(\mathcal{O}(\N_2^m)^{\GL_2})\leq 4(q-1)$, in case $q>2$.
\end{enumerate}
\end{cor}
\begin{proof} See Theorem~\ref{theo_main_finite} and formulas~\Ref{eq_zeta2}, \Ref{eq_eta2}.
\end{proof}

\begin{cor}\label{cor_main_infinite_be}
Assume that $\FF$ is infinite and $m\geq2$. Then 

\begin{enumerate}
\item[$\bullet$] $\besep(\mathcal{O}(\N_2^m)^{\GL_2})\leq 2$, in case $\Char{\FF}=2$ or $m=2$;

\item[$\bullet$] $\besep(\mathcal{O}(\N_2^m)^{\GL_2})\leq 3$, in case $\Char{\FF}\neq 2$ and $m>2$.
\end{enumerate}
\end{cor}
\begin{proof} See Theorem~\ref{theo_main_infinite}.
\end{proof}

\begin{cor}\label{cor_main_finite_si}
Assume that $\FF$ is an arbitrary field. Then 
$$\sisep(\mathcal{O}(\N_2),\GL_2)=\left\{
\begin{array}{cc}
2, & \text{ if } \Char{\FF}=2 \\
3, & \text{ if } \Char{\FF}\neq2 \\
\end{array}
\right..
$$
\end{cor}
\begin{proof}
The upper bound on $\sisep(\mathcal{O}(\N_2),\GL_2)$ follows from Theorems~\ref{theo_main_finite},~\ref{theo_main_infinite}. To obtain the lower bound on $\sisep(\mathcal{O}(\N_2),\GL_2)$ we consider $\un{A},\un{B}$ from Claims 3 and 4 of the proof of Theorem~\ref{theo_main_finite}. 
\end{proof}

\begin{cor}\label{cor_minimal}
Assume that $\FF=\FF_q$. Then a minimal separating set
\begin{enumerate}
\item[$\bullet$] $H_{2,m}^{(2)}$, in case $\Char{\FF}=2$;

\item[$\bullet$] $H_{2,m}$, in case $\Char{\FF}>2$; 
\end{enumerate}
for  $\mathcal{O}(\N_2^m)^{\GL_2}$ contains the least possible number of elements for a separating set for  $\mathcal{O}(\N_2^m)^{\GL_2}$ if and only if $m=1$ or $m=q=2$.
\end{cor}
\begin{proof}
The set from the formulation of corollary is a minimal separating set for  $\mathcal{O}(\N_2^m)^{\GL_2}$ by Theorem~\ref{theo_main_finite}. We have
$$|H_{2,m}^{(2)}|=m+\binom{m}{2}(q-1) \;\;\text{ and }\;\; |H_{2,m}|=|H_{2,m}^{(2)}|+\binom{m}{3},$$
where the binomial coefficient $\binom{m}{k}$ is zero in case $m<k$.

Assume $m=1$. Then Corollary~\ref{cor_cor} implies that $\ga=|H_{2,m}^{(2)}|=|H_{2,m}|=1$; i.e., the required is proven.

Assume $m=q=2$. Then $\Char \FF=2$ and Corollary~\ref{cor_cor} implies that $\ga=|H_{2,m}^{(2)}|=3$; i.e., the required is proven.

Assume $m=2$ and $q\geq 3$. Then Corollary~\ref{cor_cor} implies that $\ga=2$, but $|H_{2,m}|=|H_{2,m}^{(2)}|=q+1>\ga$.

Assume $m\geq3$. Then Corollary~\ref{cor_cor} implies that $\ga=2m-2$, but $|H_{2,m}|>|H_{2,m}^{(2)}|=m+\binom{m}{2}(q-1) >\ga$, since $\binom{m}{2}\geq m$.
\end{proof}
\medskip

\begin{example} Assume $\FF=\FF_2$ and $m=2$. By Theorem~\ref{theo_orb}, each $\GL_2$-orbit on $\N_2^2$ contains one and only one element from the following set:
$$(0,0),\; (E_{12},0),\; (E_{12},E_{12}),\; (0,E_{12}),\;(E_{12},E_{12}).$$
Corollary~\ref{cor_minimal} implies that the set 
$$H_{2,2}^{(2)}=\{\tr(Y_1Y_2),\, \zeta(Y_1),\, \zeta(Y_2)\}$$
is a separating set for $\mathcal{O}(\N_2^2)^{\GL_2}$ containing the least possible number of elements.
\end{example}


\end{document}